\documentclass{amsart}

\usepackage{url,amssymb,enumerate,colonequals}
\usepackage{tabu}
\usepackage{mathrsfs} 
\usepackage[section]{placeins}
\usepackage{MnSymbol}
\usepackage{extarrows}
\usepackage{lscape}
\usepackage[all,cmtip]{xy}

\usepackage[OT2,T1]{fontenc}
\usepackage{color}

\usepackage[
        colorlinks, citecolor=darkgreen,
        backref,
        pdfauthor={Katerina Santicola},
]{hyperref}

\usepackage{comment}
\usepackage{multirow}

\newcommand{\Q}{\mathbb{Q}}

\newcommand{\Z}{\mathbb{Z}}

\newcommand{\cP}{\mathscr{P}}

\DeclareMathOperator{\lcm}{lcm}

\DeclareMathOperator{\ord}{ord}

\newtheorem{thm}{Theorem}

\newtheorem{lem}[thm]{Lemma}

\newtheorem{cor}[thm]{Corollary}

\newtheorem{question}[thm]{Question}

\theoremstyle{definition}

\theoremstyle{remark}

\definecolor{darkgreen}{rgb}{0,0.5,0}

\DeclareRobustCommand{\SkipTocEntry}[5]{}

\begin{document}

\title[]{
	Reverse engineered Diophantine equations over $\mathbb{Q}$
}

\begin{abstract}
	Let $\cP_\Q=\{ \alpha^n \; : \; \alpha \in \Q, \; n \ge 2\}$ be
	the set of rational perfect powers, and let $S \subseteq \cP_\Q$
	be a finite subset. We prove the existence of a polynomial
	$f_S \in \Z[X]$ such that $f(\Q) \cap \cP_\Q=S$.
	This generalizes a recent theorem of Gajovi\'{c} who
	recently proved a similar theorem for 
	finite subsets of integer perfect powers.
	Our approach makes use of the
	resolution of the generalized Fermat
	equation of signature $(2,4,n)$ due
	to Ellenberg and others, as well as the
	finiteness of 
	perfect powers in non-degenerate binary recurrence sequences,
	proved by Peth\H{o} and by Shorey and Stewart.
\end{abstract}

\author{Katerina Santicola}

\address{Mathematics Institute\\
    University of Warwick\\
    CV4 7AL \\
    United Kingdom}

\email{Katerina.Santicola@warwick.ac.uk}

\date{\today}
\thanks{The author's research is funded by
a doctoral studentship from the Heilbronn Institute
for Mathematical Research}

\maketitle

\section{Introduction}
By an \emph{exponential Diophantine equation} 
\cite{STbook} we often mean
an equation of the form
\begin{equation}\label{eqn:exp}
	f(X)=Y^n
\end{equation}
where $f$ is a polynomial with integer or rational coefficients,
and the unknowns $X$, $Y$ are taken to be either integral or rational.
The exponent $n$ is also an unknown, usually taken to run through 
the integers $n \ge 2$. The epithet \lq\lq exponential\rq\rq\ simply
refers to the fact that the exponent is a variable (if the exponent
is fixed the equation is called superelliptic).
A famous theorem of Schinzel and Tijdeman \cite{SchinzelTijdeman}
concerning such equations asserts that if $f \in \Z[X]$
as at least two distinct roots, then there are finitely
many solutions $(X,Y,n)$ to \eqref{eqn:exp} with $X$, $Y$, $n \in \Z$, $n \ge 3$
and $\lvert Y \rvert \ne 1$.
One of the earliest examples of an exponential Diophantine equation
is
\begin{equation}\label{eqn:Lebesgue}
	X^2+1=Y^n, \qquad X,~Y,~n \in \Z, \quad n \ge 2;
\end{equation}
In 1850 Victor Lebesgue \cite{Lebesgue} proved
that the only solutions are $(X,Y)=(0,1)$ if $n$ is odd,
and $(X,Y)=(0,\pm 1)$ if $n$ is even.
Lebesgue's argument is elementary, and some version
of this argument, with the help of the Primitive Divisor
Theorem of Bilu, Hanrot and Voutier \cite{BHV} is capable
of resolving many equations of the form
$X^2+C=Y^n$ (e.g. \cite{Cohn}). In contrast, the equation $X^2+7=Y^n$ 
(which generalizes the famous Ramanujan--Nagell equation) 
apparently cannot be tackled by elementary arguments,
and was solved by
Bugeaud, Mignotte and Siksek \cite{BMS} using a combination
of tools from Diophantine approximation and Galois 
representations of elliptic curves.
We can think of the Catalan equation
\[
	X^m-Y^n=1, \qquad X,~Y,~m,~n \in \Z, \quad m,~n \ge 2
\]
as an infinite family of exponential Diophantine 
equations of the form \eqref{eqn:exp}
by rewriting it as $X^m+1=Y^n$,
giving an exponential Diophantine equation
for each value of $m$.
The famous Catalan conjecture, proved
by Mih\u{a}ilescu \cite{Mihailescu} in 2004,
states the only solution with $XY \ne 0$ is $3^2-2^3=1$.

We can view equation \eqref{eqn:exp} as the question
of which perfect powers belong to $f(\Z)$
(if the unknowns $X$, $Y$ belong to $\Z$)
or to $f(\Q)$ 
(if the unknowns $X$, $Y$ belong to $\Q$).
It is therefore natural to ask which
whether there is a restriction
on the set of perfect powers
that can belong to $f(\Z)$
or $f(\Q)$.
Indeed, at the recent 
\lq\lq Rational Points\rq\rq\ conference (Schney, April 2022),
Siksek posed the following two questions.
\begin{question}\label{question:1}
Let 
	\[
		\cP_\Z \; = \; \{ a^n \; : \; a \in \Z, \quad n \ge 2\}
	\]
be the set of perfect powers in $\Z$. Let $S$ be a finite
	subset of $\cP_\Z$. Is there a polynomial $f_S \in \Z[X]$
	such that $f_S(\Z) \cap \cP_\Z=S$?
\end{question}
\begin{question}\label{question:2}
Let 
	\begin{equation}\label{eqn:cQ}
		\cP_\Q \; = \; \{ \alpha^n \; : \; \alpha \in \Q, \quad n \ge 2\}
	\end{equation}
be the set of perfect powers in $\Q$. Let $S$ be a finite
	subset of $\cP_\Q$. Is there a polynomial $f_S \in \Q[X]$
	such that $f_S(\Q) \cap \cP_\Q=S$?
\end{question}

Question~\ref{question:1} was answered
affirmatively by  Gajovi\'{c} \cite[Theorem 3.1]{Gajovic},
and we briefly recall (and slightly simplify) his elegant argument which 
  yields an
explicit polynomial $f_S$ in terms of $S$.
Let $S=\{b_1,b_2,\dotsc,b_r\} \subseteq \cP_\Z$.
Let
\[
	g(X)\; =\;  \prod_{i=1}^r (X-b_i)^2+1,
	\qquad
	h(X) \; = \; (X-1) \cdot g(X) +1, 
	\qquad
	f_S(X) \; =\; g(X) \cdot h(X).
\]
It is clear that $f_S(b_i)=b_i$, so $S \subseteq f_S(\Z) \cap \cP_\Z$.
To prove the reverse inclusion let $x \in \Z$
such that $f_S(x) =y^n$ for some $y \in \Z$ and $n \ge 2$;
it is enough to prove that $x=b_i$ for some $i$.
Write $c=\prod_{i=1}^r(x-b_i)$.
We consider two cases:
\begin{itemize}
	\item $y=0$. Then $h(x)=0$, so $x-1=-1/(c^2+1)$.
	Since $x \in \Z$, we have $c=0$, and so $x=b_i$ for some $i$.
\item $y \ne 0$. Note that $g(x)>0$ and $h(x)$
		are coprime integers whose product
		is $y^n$. Thus
 $g(x)=z^n$ for some integer $z$,
and so $c^2+1=z^n$.
Whence $c=0$, by the aforementioned theorem of Lebesgue concerning
\eqref{eqn:Lebesgue}. Hence $x=b_i$ for some $i$,
completing the proof.
\end{itemize}
We point out that for the last step Gajovi\'{c} invokes
Mih\u{a}ilescu's theorem \cite{Mihailescu}
(the Catalan conjecture), although it is enough
to invoke the special (and elementary) case due to
Lebesgue.

\bigskip

The purpose this paper is to
give an affirmative answer to Question~\ref{question:2}.
Concisely, our theorem can be stated as follows.
\begin{thm}\label{thm:main}
Let $S$ be a finite subset of $\cP_\Q$. 
	Then there is a polynomial $f_S \in \Z[X]$
	such that $f_S(\Q) \cap \cP_\Q=S$.
\end{thm}
We give a recipe for $f_S$ in terms of a certain effectively
computable (though currently inexplicit) constant.
Instead of Lebesgue's theorem, or Mih\u{a}ilescu's theorem, 
our proof that $f_S(\Q) \cap \cP_\Q=S$ crucially relies on
the following theorem due to Ellenberg, 
to Bennett, Ellenberg and Ng,
and to Bruin.
\begin{thm}[Ellenberg et al.]\label{thm:Ellenberg}
Let $n \ge 4$. Then the equation $a^2+b^4=c^n$
has no solutions in coprime non-zero integers.
\end{thm}
The theorem is due to Bruin \cite{Bruin} for $n=6$,
who treats this case using an elliptic Chabauty argument.
It is due to Ellenberg \cite{Ellenberg24p} for prime $n \ge 211$,
and to Bennett, Ellenberg and Ng \cite{BEN}
for all other $n \ge 4$. These two papers make
use of deep results in theory of Galois representations
of $\Q$-curves as well as a careful study of 
critical values of Hecke $L$-functions of modular forms.

We shall also need a famous theorem 
on perfect powers
in non-degenerate binary recurrence sequences,
due to 
Peth\H{o} \cite{Petho} and independently to
Shorey 
and Stewart \cite{ShSt83}.
We shall not need the theorem in its full generality,
so we only state a special case.
\begin{thm}[Peth\H{o}, Shorey and Stewart]\label{thm:ST}
Let $a$, $b$, $\alpha$, $\beta$ be non-zero integers
with $\alpha \ne \pm \beta$. Let 
\[
	u_t \; = \; a \alpha^t + b \beta^t, \qquad t \in \Z, \qquad t \ge 0.
\]
Then there is an effectively
	computable constant $C(a,b,\alpha,\beta)$ 
	such that 
	$u_t \notin \cP_\Z$
	for $t > C(a,b,\alpha,\beta)$.
\end{thm}
The theorem is proved using 
the theory of lower bounds for linear forms
in logarithms of algebraic numbers.

\section{Preliminary results}
We shall need the following consequence of Theorem~\ref{thm:ST}.
\begin{lem}\label{lem:prelim}
Let $\gamma \in \Q$, $\gamma \ne 0$. Then there is an effectively 
	computable constant $D(\gamma)$
	such for $t > D(\gamma)$,
	\[
		\gamma - 2^t \notin \cP_\Q.
	\]
\end{lem}
\begin{proof}
Write $\gamma=u/v$ where $u$, $v \in \Z$, $u \ne 0$, $v \ge 1$
and $\gcd(u,v)=1$. We let 
	\[
		D(\gamma) \; =\; \max\{0,\log_2{\lvert \gamma \rvert},
		\;
		C(u,-v,1,2)
		\}
	\]
	where $C(a,b,\alpha,\beta)$ is as in Theorem~\ref{thm:ST}.
	Note that since $C(a,b,\alpha,\beta)$ is
	effectively computable, so is $D(\gamma)$.
	Let $t > D(\gamma)$. 
	We suppose $\gamma -2^t \in \cP_\Q$ and derive
	a contradiction.

	Since $t>\log_2{|\gamma|}$, we have $\gamma-2^t \ne 0$.
	Observe that $\gamma-2^t=(u-2^t v)/v$
	where the numerator $u-2^t v$
	and the denominator $v$ are coprime. As $\gamma -2^t \in \cP_\Q \setminus \{0\}$
we conclude that $u-2^t v \in \cP_\Z$.
	We apply Theorem~\ref{thm:ST}
	with $a=u$, $b=-v$, $\alpha=1$, $\beta=2$.
	Since $t>C(u,-v,1,2)$, the theorem tells us
	that $u-2^t v \notin \cP_\Z$ giving a contradiction.
\end{proof}

We also need the following two corollaries to 
Theorem~\ref{thm:Ellenberg}.
\begin{cor}\label{cor:Ellenberg1}
Let $n \ge 2$. Then the only solutions
to the equation $A^4+B^4=2C^n$ with $A$, $B$, $C \in \Z$
and $\gcd(A,B)=1$ satisfy $A=\pm 1$, $B=\pm 1$.
\end{cor}
\begin{proof}
Note that $A$, $B$ are both odd. Write
\[
	U \; =\; AB, \qquad V \; =\; \frac{A^4-B^4}{2}.
\]
Then $U$, $V$ are coprime integers and satisfy
\[
	U^4+V^2 \; = \; \left(\frac{A^4+B^4}{2} \right)^2 \; = \; C^{2n}.
\]
	By Theorem~\ref{thm:Ellenberg} we have $UVC=0$.
	However, $U$ is odd so $U \ne 0$. Thus $C \ne 0$.
	Hence $V=0$, so $A^4=B^4$. Since $A$, $B$
	are coprime, $A=\pm 1$, $B=\pm 1$ as required.
\end{proof}

\begin{cor}\label{cor:Ellenberg2}
Let $n \ge 2$. Then the only solutions
to the equation $A^4+B^4=C^n$ with $A$, $B$, $C \in \Z$
and $\gcd(A,B)=1$ satisfy $A=0$, $B=\pm 1$, or $B=0$, $A=\pm 1$.
\end{cor}
\begin{proof}
For $n=2$ this is a famous result of Fermat, proved by infinite descent.
	For $n=3$ it is in fact a result of Lucas 
	\cite[Section 5]{BCDY}.
	Suppose $n \ge 4$. Then we can rewrite the
	equation as $(A^2)^2+B^4=C^n$
	and conclude from Theorem~\ref{thm:Ellenberg}.
\end{proof}

\section{Proof of Theorem~\ref{thm:main}}\label{sec:proof}
Let $S$ be a finite subset of  $\cP_\Q$. 
We would like to give a polynomial $f_S \in \Z[X]$
such that $f_S(\Q) \cap \cP_\Q=S$.
For $S=\emptyset$
we may take any constant polynomial whose value is not a perfect power. Thus we may suppose $S \ne \emptyset$. 
We write
\[
	S\; = \; \{ \beta_1,\dotsc,\beta_r \}
\]
where $r \ge 1$, and $\beta_i \in \cP_\Q$. 
As the $\beta_i$ are 
rational numbers, we can write
\[
	\beta_i \; = \; \frac{a_i}{c_i}, \qquad
a_i,~c_i \in \Z, \qquad c_i \ge 1,  \qquad \gcd(a_i,c_i)=1.
\]
Let $p_1,p_2,\dotsc,p_t$ be the distinct prime
divisors of $\prod_{i=1}^nc_i$, and let
\begin{equation}\label{eqn:kdef}
	k\; = \; \lcm(4,p_1 -1, p_2-1, \dotsc,p_t-1).
\end{equation}
Let
\begin{equation}\label{eqn:F}
	F(X) \; =\; \prod_{i=1}^r (c_i X-a_i)^2 \; -1.
\end{equation}
Since $c_i \ge 1$ we note that $F$ has degree $2r$,
and hence at most $2r$ roots. We are only interested
in the non-zero rational roots of $F$,  and we let these 
be  $\delta_1,\delta_2,\dotsc,\delta_m$.
We let $s$ be an integer satisfying the following:
\begin{gather}
	\text{$s \, =\, 2^{\kappa}-1$ for some $\kappa \ge 1$},  \label{eqn:sdef1}\\
	s \; \ge \; D(4\delta_j), \qquad  j=1,\dotsc,m.
\label{eqn:sdef3}
\end{gather}
Here $D(\cdot)$ is as in Lemma~\ref{lem:prelim}.
It is clear that such an $s$ is effectively computable.
Let 
\begin{gather}
	g(X) \; =\; \prod_{i=1}^r (c_i X-a_i)^k \; + \; 1,
	\qquad
	h(X) \; = \; (X-2^s)g(X) \, + \, 2^s, \label{eqn:ghdef} \\
	f_S(X) \; = \; g(X) \cdot h(X) \; \in \Z[X].
	\label{eqn:fSdef}
\end{gather}
To prove Theorem~\ref{thm:main}
we need to show that $f_S(\Q) \cap \cP_\Q=S$.
Observe that $g(\beta_i)=1$, $h(\beta_i)=\beta_i$
and so $f_S(\beta_i)=\beta_i$.
Thus $\beta_i \in f_S(\Q) \cap \cP_\Q$, and therefore
$S \subseteq f_S(\Q) \cap \cP_\Q$.
We would like to prove the reverse inclusion.
Thus let $x \in \Q$ and suppose that $f_S(x) \in \cP_\Q$.
To prove Theorem~\ref{thm:main} it will
be enough to show that $x=\beta_i$ for some $1 \le i \le r$.

Write
\begin{equation}\label{eqn:uvdef}
	x=\frac{u}{v}, \qquad u,~v \in \Z, \quad v \ge 1, \quad
	\gcd(u,v)=1.
\end{equation}
Let
\begin{equation}\label{eqn:ABdef}
	A\; =\; \prod_{i=1}^r (c_i u - a_i v), \qquad
	B\; = \; \frac{A}{\gcd(A,v^r)}, \qquad w \; = \; \frac{v^r}{\gcd(A,v^r)},
\end{equation}
and note that 
\begin{equation}\label{eqn:gcdBw}
	\gcd(B,w)=1, \qquad w \ge 1.
\end{equation}
Recall that we would like to show that $x=\beta_i$ 
(or equivalently $u/v=a_i/c_i$) for some $1 \le i \le r$.
This is equivalent to $A=0$, which is equivalent to 
$B=0$. Thus our proof will be complete on showing that $B=0$.

We note that
\begin{equation}\label{eqn:gh}
	g(x) \; = \; \frac{A^k + v^{kr}}{v^{kr}} \; = \; 
	\frac{B^k+w^k}{w^k},
\qquad
	h(x) \; = \; \frac{(u-2^s v) \cdot (B^k+w^k) + 2^s v w^k}{v w^k}.
\end{equation}
\begin{lem}\label{lem:2pow}
$\gcd(B^k+w^k,v)$ is a power of $2$.
\end{lem}
\begin{proof}
Let $p$ be an odd prime dividing both $B^k+w^k$ and $v$.
	From \eqref{eqn:ABdef} we see that $(B^k+w^k)$
	divides $A^k+v^{kr}$. Thus $p \mid (A^k+v^{kr})$
	and so $p \mid A$. 
	From the product definition of $A$ in \eqref{eqn:ABdef},
	and since $\gcd(u,v)=1$ from \eqref{eqn:uvdef}, we see that $p \mid c_i$
	for some $1 \le i \le r$. It follows from the definition of $k$ in \eqref{eqn:kdef}
	that $(p-1) \mid k$. Since $\gcd(B,w)=1$ \eqref{eqn:gcdBw}, we have that $p\nmid w$ and $p\nmid B$, so $B,w\in \mathbb{Z}_p^*$. Then we have $B^k+w^k\equiv 0\mod p$, so $(Bw^{-1})^k\equiv -1\mod p$. This contradicts Fermat's little theorem.
\end{proof}

\begin{lem}\label{lem:not2}
$B^k+w^k \ne 2$.
\end{lem}
\begin{proof}
Recall that $k$ is an even integer by \eqref{eqn:kdef},
and that $w \ge 1$ from \eqref{eqn:gcdBw}.
Suppose $B^k+w^k=2$. 
Then $\lvert B \rvert=w=1$.
In particular, from \eqref{eqn:gh} we have $g(x)=2$.
Thus 
\begin{equation}\label{eqn:fS0}
	f_S(x)\; = \; g(x) \cdot h(x) \; = \; 4x-2^{s+1}
\end{equation}
	from \eqref{eqn:ghdef} and \eqref{eqn:fSdef}.
First we show that $x \ne 0$. 
	Recall our assumption that $f_S(x) \in \cP_\Q$,
	and so $f_S(x)=y^n$ for some rational $y$
	and some integer $n \ge 2$.
	If $x=0$, then $y^n=-2^{s+1}=-2^{2^\kappa}$
	by \eqref{eqn:sdef1}. Thus $n \mid 2^{\kappa}$
	and so $n$ is even. Hence $y^n >0$
	giving a contradiction. We conclude that $x \ne 0$.

	Recall, from \eqref{eqn:uvdef}, that $v \ge 1$
	and $x=u/v$. From \eqref{eqn:ABdef}, as $\lvert B \rvert=w=1$,
we see that 
\[
	\lvert A \rvert \; = \; \gcd(A,v^r) \; = \; v^r.
\]
Thus $A = \pm v^r$.  
Dividing both sides of $A = \pm v^r$ by $v^r$ gives
\[
	\prod_{i=1}^r (c_i x- a_i) \; = \; \frac{A}{v^r} \; =\; \pm 1,
\]
	from the product expression for $A$ in \eqref{eqn:ABdef}.
Thus $x$ is a non-zero root of the polynomial $F(X)$
given in \eqref{eqn:F}. We have previously
labelled the roots of $F$ by $\delta_1,\delta_2,
\dotsc,\delta_m$. Thus $x=\delta_j$ for some $1 \le j \le m$.
	By \eqref{eqn:sdef3} we have
	$s+1 > D(4 \delta_j)$. 
	Hence, by Lemma~\ref{lem:prelim}
	and \eqref{eqn:fS0} we have
	$f_S(x) =4x-2^{s+1}=4 \delta_j-2^{s+1} \notin \cP_\Q$,
	giving a contradiction.
\end{proof}

\begin{lem}\label{lem:zero}
If $f_S(x)=0$ then $B=0$.
\end{lem}
\begin{proof}
	Suppose $f_S(x)=0$ and recall that $f_S(x)=g(x)h(x)$.
	Note that $k$ is even by \eqref{eqn:kdef},
	and also $w \ge 1$ by \eqref{eqn:gcdBw}.
	Thus $B^k+w^k$ is a positive integer
	and hence, from \eqref{eqn:gh}
	we have $g(x) \ne 0$.
	Hence $h(x)=0$,
	whence from \eqref{eqn:gh},
	\begin{equation}\label{eqn:-vw}
		(u-2^s v) \cdot (B^k+w^k) \; = \; -2^s v w^k.
	\end{equation}
	We claim that $B^k+w^k$ is a power of $2$. To prove
	this claim let $p \mid (B^k+w^k)$ be an odd prime.
	As $\gcd(B,w)=1$ by \eqref{eqn:gcdBw},
	we have $p \nmid w$.
	From Lemma~\ref{lem:2pow} we note that $p \nmid v$.
	However, from \eqref{eqn:-vw}, we have $p \mid 2^s v w^k$
	giving a contradiction.
	Thus $B^k+w^k$ is indeed a power of $2$.
	Since $k$ is even and $\gcd(B,w)=1$ we see
	that $4 \nmid (B^k+w^k)$ and so
	$B^k+w^k=1$ or $2$. 
	However, $B^k+w^k \ne 2$
	from Lemma~\ref{lem:not2}.
	We conclude that $B^k+w^k=1$.
	As $w \ge 1$, we obtain $w=1$ and $B=0$
	as required.
\end{proof}

Recall that we have supposed that $f_S(x) \in \cP_\Q$
and to complete the proof of Theorem~\ref{thm:main},
it is enough to show that $B=0$.
Lemma~\ref{lem:zero} establishes this if $f_S(x)=0$.
Thus we may suppose $f_S(x) \ne 0$. 
Hence $f_S(x)=y^n$ where $y$ is a non-zero rational and $n \ge 2$.
We claim that $B^k+w^k=z^n$ or $2 z^n$ for some odd positive 
integer $z$. To prove this let $p$ be an odd prime dividing 
$B^k+w^k$. Then, as before $p \nmid w$ since $\gcd(B,w)=1$, 
and $p\nmid v$ from Lemma~\ref{lem:2pow}.
From the expression for $h(x)$ in \eqref{eqn:gh}
we see that $p$ divides neither the numerator nor the denominator of $h(x)$,
and hence $\ord_p(h(x))=0$. 
However $g(x)h(x)=f_S(x)=y^n$, so
\[
	\ord_p(B^k+w^k)
	\; = \; \ord_p\left(\frac{B^k+w^k}{w^k} \right)\; =\; \ord_p(g(x)) \; = \; \ord_p(y^n) \; \equiv\; 0 \pmod{n}.
\]
As this is true for every odd prime dividing $B^k+w^k$
we have $B^k+w^k=2^e z^n$ for some $e \ge 0$ and some odd integer $z$.
Now as before $4 \nmid (B^k+w^k)$, since $k$ is
even and $\gcd(B,w)=1$. Thus $B^k+w^k=z^n$
or $B^k+w^k=2 z^n$ where $z$ is odd.

Suppose first that $B^k+w^k=2z^n$. 
Recall that $4 \mid k$ by \ref{eqn:kdef}.
By Corollary~\ref{cor:Ellenberg1} we have $B=\pm 1$, $w=1$,
contradicting Lemma~\ref{lem:not2}.

Thus $B^k+w^k=z^n$. We apply Corollary~\ref{cor:Ellenberg2}
to conclude that $B=0$ or $w=0$. However, $w \ge 1$,
so $B=0$ completing the proof.

\end{document}